\documentclass[reqno]{amsart}
\usepackage[american]{babel}
\usepackage{amsmath,amsfonts,amsthm,amssymb,latexsym,euscript}
\numberwithin{equation}{section}

\newtheorem{theorem}{Theorem}[section]
\newtheorem{lemma}{Lemma}[section]

\theoremstyle{definition}
\newtheorem{definition}{Definition}[section]
\theoremstyle{remark}
\newtheorem{remark}{Remark}[section]

\DeclareMathOperator{\sgn}{sign}

\begin{document}

\title[Zakharov--Kuznetsov equation]
{On initial-boundary value problems in a strip for generalized two-dimensional Zakharov--Kuznetsov equation}

\author[A.V.~Faminskii]{Andrei~V.~Faminskii}
\author[E.S.~Baykova]{Evgeniya~S.~Baykova}


\address{Department of Mathematics, Peoples' Friendship University of Russia,
Miklukho--Maklai str. 6, Moscow, 117198, Russia}
\email{afaminskii@sci.pfu.edu.ru}
\email{baykova82@inbox.ru}

\date{}
\maketitle

{\scriptsize \centerline{Peoples' Friendship University of Russia, Moscow, Russia}}

\begin{abstract}
Initial-boundary value problems in a strip with different types of boundary conditions for two-dimensional generalized Zakharov--Kuznetsov equation are considered. Results on global existence and uniqueness of weak solutions in certain weighted spaces are established.
\end{abstract}

\section{Introduction. Description of main results}\label{S1}

Zakharov--Kuznetsov equation (ZK) on the plane is written as follows:
$$
u_t+u_{xxx}+u_{xyy}+uu_x=0.
$$

It has been derived in \cite{ZK} for description of ion-acoustic wave processes in plasma put in the magnetic field. Further, this equation has been considered as a model equation for  non-linear waves propagating in dispersive media in the preassigned direction $(x)$ and deformed in the transverse direction $(y)$. Zakharov--Kuznetsov equation is one of the variations of multi-dimensional generalizations for Korteweg--de~Vries equation (KdV) $u_t+u_{xxx}+uu_x=0$.

Boundary value problems for this equation (and its certain generalizations) were usually studied before in domains, which were products of an interval (bounded or non-bounded) on the variable $x$ and the whole line on the variable $y$. In particular, there were the initial value problem (\cite{S1, F89, F95, LP}), as well as initial-boundary value problems in $\mathbb R_+\times \mathbb R, \mathbb R_-\times \mathbb R,$ and $I\times \mathbb R,$ where $I$ is a bounded interval, (\cite{FB1, FB2, F08, ST, F12} and others). On the other hand, it seems more natural from the physical point of view to consider domains, where the variable $y$ varies in a bounded interval.

In the present paper we consider initial-boundary value problems in a layer $\Pi_T=(0,T)\times\Sigma$, where $\Sigma=\mathbb R\times (0,L)=\{(x,y): x\in \mathbb R, 0<y<L\}$ is a strip of a given width $L$, for generalized Zakharov--Kuznetsov equation
\begin{equation}\label{1.1}
u_t+u_{xxx}+u_{xyy}+(g(u))_x=f(t,x,y)
\end{equation}
with an initial condition 
\begin{equation}\label{1.2}
u(0,x,y)=u_0(x,y),\qquad (x,y)\in\Sigma,
\end{equation}
and boundary conditions of one of the following four types: 
for $(t,x)\in(0,T)\times\mathbb R$
\begin{equation}\label{1.3}
\begin{split}
\mbox{whether}\qquad &a)\mbox{ } u(t,x,0)=u(t,x,L)=0,\\
\mbox{or}\qquad &b)\mbox{ } u_y(t,x,0)=u_y(t,x,L)=0,\\
\mbox{or}\qquad &c)\mbox{ } u(t,x,0)=u_y(t,x,L)=0,\\
\mbox{or}\qquad &d)\mbox{ } u \mbox{ is an  $L$-periodic function with respect to $y$.}
\end{split}
\end{equation}

We use the notation "problem \eqref{1.1}--\eqref{1.3}" for each of these four cases.

The main results of this paper are theorems on existence and uniqueness of global weak solutions ($T>0$ -- arbitrary). We introduce certain growth restrictions on the function $g$ as $|u|\to\infty$ which are true for ZK equation itself.

The research of global well-posedness is based, firstly, on conservation laws for equation \eqref{1.1} (with $f\equiv 0$):
\begin{equation}\label{1.4}
\iint_{\Sigma}u^2\,dxdy=\text{const}, \mbox{  }\iint_{\Sigma}\left(u_x^2+u^2_y-g^{*}(u)\right)\,dxdy=\text{const}, 
\end{equation}
where here and further $g^{*}(u)\equiv\int_0^u g(\theta)\,d\theta$ is the primitive for $g$, which are true for each of the four cases of the boundary conditions, and secondly, on the effect of local smoothing, meaning that solutions have derivatives with respect to space variables of orders greater by one, than initial functions. For example, if $u_0\in L_2(\Sigma)$, then
$$
\int_0^T\int_{-r}^r\int_0^L(u^2_x+u^2_y)\,dydxdt<\infty
$$
for any $r>0$. Such an effect of local smoothing for the initial value problem was firstly discovered in \cite{KF, K} for KdV.

The result, based on conservation laws (1.4), on existence of global solution to the initial value problem, where $u_0\in H^1(\mathbb R^2)$, was obtained in \cite{S1} (where  equations of more general type than \eqref{1.1} were considered). Existence of global solution to the initial value problem for $u_0\in L_2(\mathbb R^2)$, based on the first conservation law \eqref{1.4} and on the effect of local smoothing, was obtained in \cite{F89} (were equations of more general type were considered too). However, in the study of uniqueness there were the growth restrictions on the function $g$, which excluded ZK equation itself.

Classes of global well-posedness for the initial-value problem for ZK equation, where the initial function $u_0\in H^k(\mathbb R^2)$ for each natural $k$, were constructed in \cite{F95} on the base of the ideas on more accurate study of properties of the linear part of the equation elaborated earlier in \cite{KPV} for KdV. Similar results for modified ZK equation ($g(u)=u^3/3$) were obtained in \cite{LP}.

The attempt of studying properties of the linear part of ZK equation for the initial-boundary value problem in the strip $\Sigma$ with periodic boundary conditions was made in \cite{LPS}, but the established estimates allowed to prove only local well-posednesss in the spaces $H^s(\Sigma)$ for $s>3/2$.

The recent paper \cite{LT}, where an initial-boundary value problem for ZK equation in a half-strip $\mathbb R_+\times (0,L)$ with zero Dirichlet boundary conditions was considered, should be also mentioned. The usage of exponential weights as $x\to +\infty$ allowed to prove there global well-posedness for such a problem in smooth function spaces.

Introduce now the following notation to describe the main results of the present paper. For an integer $k\geq 0$ let
$$
|D^k\varphi|=\Bigl(\sum_{k_1+k_2=k}(\partial^{k_1}_x\partial_y^{k_2}\varphi)^2\Bigr)^{1/2}, \qquad
|D\varphi|=|D^1\varphi|.
$$
Let $L_p=L_p(\Sigma)$, $H^k=H^k(\Sigma)$, $x_+=\max(x,0)$, $x_-=\max(-x,0)$, $\mathbb R_+=(0,+\infty)$, $\mathbb R_-=(-\infty,0)$, $\Sigma_{\pm}=\mathbb R_\pm\times (0,L)$, $\Pi^\pm_T=(0,T)\times\Sigma_{\pm}$.

For any $\alpha\geq 0$ define function spaces
$$
L_2^\alpha=L_2^\alpha(\Sigma)=\{\varphi\in L_2: (1+x_+)^\alpha\varphi\in L_2\},
$$
$$
H^{k,\alpha}=H^{k.\alpha}(\Sigma)=\{\varphi\in H^k: |D^j\varphi|\in L_2^\alpha, \  j=0,\dots,k\}
$$
with natural norms (here $H^{0,\alpha}=L_2^\alpha$).

We shall construct solutions to the considered problems in spaces $X^{k,\alpha}(\Pi_T)$, $k=0 \mbox{ or }1$, consisting of functions $u(t,x,y)$ such that 
$$
u\in C_w([0,T]; H^{k,\alpha}), \qquad \sup_{x_0\in\mathbb R}\int_0^T\! \int_{x_0}^{x_0+1}\! \int_0^L |D^{k+1}u|^2\,dydxdt<\infty
$$
(symbol $C_w$ denotes the space of  weakly continuous mappings) and if $\alpha>0$ then additionally
$$
(1+x)^{\alpha-1/2}|D^{k+1}u|\in L_2(\Pi_T^+)
$$
(let $X^\alpha(\Pi_T)=X^{0,\alpha}(\Pi_T)$).

\begin{theorem}\label{T1.1}
Let $g\in C^1(\mathbb R)$ and for certain constants $c\geq 0$ and $b\in [1,2)$
\begin{equation}\label{1.5}
|g'(u)|\leq c(1+|u|^b)\qquad \forall u\in\mathbb R.
\end{equation}
Let $u_0\in L_2^\alpha$, $f\in L_1(0,T; L_2^\alpha)$ for certain $\alpha\geq 0$ and $T>0$.
Then there exists a weak solution to each of problems \eqref{1.1}--\eqref{1.3} in the space $X^\alpha(\Pi_T)$.
\end{theorem}

\begin{theorem}\label{T1.2}
Let $g\in C^2(\mathbb R)$ and for certain constants $c\geq 0$ and $b\in [1,2)$
\begin{equation}\label{1.6}
|g''(u)|\leq c(1+|u|^{b-1})\qquad \forall u\in\mathbb R.
\end{equation}
Let $u_0\in H^{1,\alpha}$, $f\in L_1(0,T; H^{1,\alpha})$ for certain $\alpha\geq0$ and $T>0$,
\begin{align*}
u_0|_{y=0}=u_0|_{y=L}=0,\mbox{  }f|_{y=0}=f|_{y=L}=0\quad &\mbox{in the case $a$},\\
u_0|_{y=0}=0, f|_{y=0}=0\quad &\mbox{in the case $c$},\\
u_0|_{y=0}=u_0|_{y=L},\mbox{  }f|_{y=0}=f|_{y=L}\quad &\mbox{in the case $d$}.
\end{align*}
Then there exists a weak solution to each of problems \eqref{1.1}--\eqref{1.3} in the space $X^{1,\alpha}(\Pi_T)$.  If $\alpha\geq 1/2$ then the solution is unique in this space.
\end{theorem}

Similar results for the initial value problem for generalized KdV equation were earlier obtained in \cite{F88}.

Since the constructed solutions are only weak there naturally arises a problem on solubility of the considered problems in spaces of smoother functions. This problem remains open, and one of the obstacles to resolve it is the absence of conservation laws other than \eqref{1.4} in smoother spaces in comparison, for example, with KdV equation.

Further we use the following auxiliary functions. Let $\eta(x)$ denotes a cut-off function, namely, $\eta$ is an infinitely smooth non-decreasing on $\mathbb R$ function such that $\eta(x)=0$ when $x\leq 0$, $\eta(x)=1$ when $x\geq 1$, $\eta(x)+\eta(1-x)\equiv 1$.

We say that $\rho(x)$ is an admissible weight function if $\rho$ is an infinitely smooth positive on $\mathbb R$ function such that $|\rho^{(j)}(x)|\leq c(j)\rho(x)$ for each natural $j$ and all $x\in\mathbb R$.

For each $\alpha\geq 0$ and $\beta>0$ we introduce an infinitely smooth increasing on $\mathbb R$ function $\rho_{\alpha,\beta}(x)$ as follows: 
$\rho_{\alpha,\beta}(x)=e^{\beta x}$ when $x\leq -1$, $\rho_{\alpha,\beta}(x)=(1+x)^\alpha$ for $\alpha>0$ and $\rho_{0,\beta}(x)=2-(1+x)^{-1/2}$ when $x\geq 0$, $\rho'_{\alpha,\beta}(x)>0$ when $x\in (-1,0)$. 

Note that both $\rho_{\alpha,\beta}$ and $\rho'_{\alpha,\beta}$ are admissible weight functions, and $\rho'_{\alpha,\beta}(x)\leq c(\alpha,\beta)\rho_{\alpha,\beta}(x)$ for all $x\in \mathbb R$.

Note also that if $u\in X^{k,\alpha}(\Pi_T)$ for $\alpha\geq 1/2$, then $|D^{k+1}u|\rho_{\alpha-1/2,\beta}(x)\in L_2(\Pi_T)$ for any $\beta>0$.

Further we need the following interpolating inequality.

\begin{lemma}\label{L1.1}
Let $\rho_1(x)$, $\rho_2(x)$ be two admissible weight functions such that $\rho_1(x)\leq c_0\rho_2(x)$ for some constant $c_0>0$. Let $k=1$ or $2$, $m\in [0,k)$ -- integer, $q\in [2,+\infty]$ if $k=2, m=0$ and $q\in [2,+\infty)$ in other cases. For the case $q=+\infty$  assume also that $\displaystyle{\frac{\rho_2(x_1)}{\rho_1(x_1)}\leq c_0\frac{\rho_2(x_2)}{\rho_1(x_2)}}$ if $|x_1-x_2|\leq 1$.
Then there exists a constant $c>0$ such that for every function $\psi(x,y)$, which satisfies  assumptions $|D^k\psi|\rho_1^{1/2}(x)\in L_2$, $\psi\rho_2^{1/2}(x)\in L_2$, the following inequality holds
\begin{equation}\label{1.7}
\bigl\| |D^m\psi|\rho_1^s(x)\rho_2^{1/2-s}(x)\bigr\|_{L_q} \leq c 
\bigl\| |D^k\psi|\rho_1^{1/2}(x)\bigr\|^{2s}_{L_2}
\bigl\| \psi\rho_2^{1/2}(x)\bigr\|^{1-2s}_{L_2} 
+\bigl\| \psi\rho_2^{1/2}(x)\bigr\|_{L_2},
\end{equation}
where $\displaystyle{s=\frac{m+1}{2k}-\frac{1}{kq}}$.
\end{lemma}

\begin{proof}
If one considers the whole plane $\mathbb R^2$ instead of the strip $\Sigma$ the given inequality is a special case for a more general interpolating inequality, estimated in \cite{F89} for an arbitrary number of variables and arbitrary values of $k$. The proof in this case is similar.
\end{proof}

As a rule further we omit limits of integration in integrals over the whole strip $\Sigma$.

The paper is organized as follows. An auxiliary linear problem is considered in Section~\ref{S2}. Section~\ref{S3} is dedicated to problems on existence of solutions to the original problem. Results on continuous dependence of solutions on $u_0$ and $f$ are proved in Section~\ref{S4}. In particular, they imply uniqueness of the solution.

\section{An auxiliary linear equation}\label{S2}

Consider a linear equation 
\begin{equation}\label{2.1}
u_t+u_{xxx}+u_{xyy}-\delta(u_{xx}+u_{yy})=f(t,x,y)
\end{equation}
for a certain constant $\delta\in [0,1]$.

Introduce certain additional function spaces. Let $\EuScript S(\overline{\Sigma})$ be a space of infinitely smooth in $\overline{\Sigma}$ functions $\varphi(x,y)$ such that $\displaystyle{(1+|x|)^n|\partial^k_x\partial^l_y\varphi(x,y)|\leq c(n,k,l)}$ for any integer non-negative $n, k, l$ and all $(x,y)\in \overline{\Sigma}$. Let $\EuScript S_{exp}(\overline{\Sigma}_{\pm})$ denote a space of infinitely smooth in $\overline{\Sigma}_{\pm}$ functions $\varphi(x,y)$ such that $e^{n|x|}|\partial ^k_x\partial^l_y\varphi(x,y)|\leq c(n,k,l)$ for any integer non-negative $n, k, l$ and all $(x,y)\in \overline{\Sigma}_\pm$.

\begin{lemma}\label{L2.1}
Let $u_0\in \EuScript S(\overline{\Sigma})\cap \EuScript S_{exp}(\overline{\Sigma}_+)$, $f\in C^\infty\bigl([0,T]; \EuScript S(\overline{\Sigma})\cap \EuScript S_{exp}(\overline{\Sigma}_+)\bigr)$ and for any integer $j\geq 0$
\begin{align*}
\partial^{2j}_y u_0|_{y=0}=\partial^{2j}_y u_0|_{y=L}=0, &\mbox{  } \partial^{2j}_y f|_{y=0}=\partial^{2j}_y f|_{y=L}=0 &\mbox{  in the case  }a,\\
\partial^{2j+1}_y u_0|_{y=0}=\partial^{2j+1}_y u_0|_{y=L}=0, &\mbox{  } \partial^{2j+1}_y f|_{y=0}=\partial^{2j+1}_y f|_{y=L}=0 &\mbox{  in the case }b,\\
\partial^{2j}_y u_0|_{y=0}=\partial^{2j+1}_y u_0|_{y=L}=0, &\mbox{  }\partial^{2j}_y f|_{y=0}=\partial^{2j+1}_y f|_{y=L}=0 &\mbox{   in the case  }c,\\
\partial^{j}_y u_0|_{y=0}=\partial^{j}_y u_0|_{y=L}, &\mbox{  }\partial^{j}_y f|_{y=0}=\partial^{j}_y f|_{y=L} &\mbox{   in the case  }d.
\end{align*}
Then there exists a unique solution to each of problems \eqref{2.1}, \eqref{1.2}, \eqref{1.3} $u\in C^\infty\bigl([0,T]; \EuScript S(\overline{\Sigma})\cap \EuScript S_{exp}(\overline{\Sigma}_+)\bigr)$.
\end{lemma}

\begin{proof}
Let $\psi_l(y)$, $l=1,2\dots$, be the orthonormal in $L_2(0,L)$ system of the eigenfunctions for the operator $(-\psi'')$ on the segment $[0,L]$ with corresponding boundary conditions   $\psi(0)=\psi(L)=0$ in the case a, $\psi'(0)=\psi'(L)=0$ in the case b, $\psi(0)=\psi'(L)=0$ in the case c, $\psi(0)=\psi(L),\psi'(0)=\psi'(L)$ in the case d, $\lambda_l$ be the corresponding eigenvalues.    
Such systems are well-known and can be written in trigonometric functions. Then with the use of Fourier transform for the variable $x$ and Fourier series for the variable $y$  a solution to  problem \eqref{2.1}, \eqref{1.2}, \eqref{1.3} can be written as follows:
$$
u(t,x,y)=\frac{1}{2\pi}\int\limits_{\mathbb R}\sum_{l=1}^{+\infty}e^{i\xi x}\psi_l(y)\widehat{u}(t,\xi,l)\,d\xi,
$$
where
$$
\widehat{u}(t,\xi,l)=\widehat{u_0}(\xi,l)e^{(\xi^3+\xi\lambda_l
-\delta(\xi^2+\lambda_l))t} \\+
\int_0^t\widehat{f}(\tau,\xi,l)e^{(\xi^3+\xi\lambda_l-\delta(\xi^2+\lambda_l))(t-\tau)}\,d\tau,
$$
$$
\widehat{u_0}(\xi,l)\equiv\iint e^{-i\xi x}\psi_l(y)u_0(x,y)\,dxdy,
$$
$$
\widehat f(t,\xi,l)\equiv\iint e^{-i\xi x}\psi_l(y) f(t,x,y)\,dxdy.
$$
According to the properties of the functions $u_0$ and $f$ this solution $u\in C^\infty([0,T],\EuScript S(\overline{\Sigma}))$.

Next, let $v\equiv \partial^k_x\partial^l_y u$ for some $k, l\geq 0$. Then the function $v$ satisfies an equation of \eqref{2.1} type, where $f$ is replaced by $\partial^k_x\partial^l_y f$. Let $m\geq 3$. Multiplying this equation by $2x^m v$ and integrating over $\Sigma_+$, we derive an inequality 
\begin{multline*}
\frac{d}{dt}\iint_{\Sigma_+}x^mv^2\, dxdy \leq m(m-1)(m-2)\iint_{\Sigma_+}x^{m-3}v^2 \,dxdy \\
+\delta m(m-1)\iint_{\Sigma_+}x^{m-2}v^2\,dxdy +2\iint_{\Sigma_+}\partial^k_x\partial^l_y fv\,dxdy.
\end{multline*}
Let $\alpha>0$, $n\geq3$. For any $m\in [3,n]$ multiplying the corresponding inequality by $\alpha^m/(m!)$ and summing by $m$ we obtain that for
$$
z_n(t)\equiv \iint_{\Sigma_+}\sum_{m=0}^n\frac{(\alpha x)^m}{m!}v^2(t,x,y)\,dxdy
$$
inequalities
$$
z_n'(t)\leq c z_n(t)+c, \quad z_n(0)\leq c,
$$
are valid uniformly with respect to $n$,
whence it follows that
$$
\sup_{t\in[0,T]}\iint_{\Sigma_+}e^{\alpha x} v^2\,dxdy<\infty.
$$
Thus, $u\in C^\infty([0,T], \EuScript S_{exp}(\overline{\Sigma}_+))$. 
\end{proof}

We now turn to generalized solutions. Let $u_0\in \bigl(\EuScript S(\overline{\Sigma})\cap \EuScript S_{exp}(\overline{\Sigma}_-)\bigr)'$, $f \in \bigl(C^\infty([0,T]; \EuScript S(\overline{\Sigma})\cap \EuScript S_{exp}(\overline{\Sigma}_-))\bigr)'$.

\begin{definition}\label{D2.1}
A function $u \in \bigl(C^\infty([0,T]; \EuScript S(\overline{\Sigma})\cap \EuScript S_{exp}(\overline{\Sigma}_-))\bigr)'$ is called a generalized solution to corresponding problem \eqref{2.1}, \eqref{1.2}, \eqref{1.3}, if for any function $\varphi \in C^\infty\bigl([0,T]; \EuScript S(\overline{\Sigma})\cap \EuScript S_{exp}(\overline{\Sigma}_-)\bigr)$ such that $\varphi|_{t=T}=0$ and $\varphi|_{y=0}=\varphi|_{y=L}=0$ for the case $a$, $\varphi_y|_{y=0}=\varphi_y|_{y=L}=0$ for the case $b$, $\varphi|_{y=0}=\varphi_y|_{y=L}=0$ for the case $c$, $\varphi|_{y=0}=\varphi|_{y=L}$, $\varphi_y|_{y=0}=\varphi_y|_{y=L}$ for the case $d$, there holds the following equality:
\begin{equation}\label{2.2}
\langle u,\varphi_t+\varphi_{xxx}+\varphi_{xyy}+\delta\varphi_{xx}+\delta\varphi_{yy}\rangle +
\langle f,\varphi\rangle +\langle u_0,\varphi|_{t=0}\rangle=0.
\end{equation}
\end{definition}

\begin{lemma}\label{L2.2}
A generalized solution to each problem \eqref{2.1}, \eqref{1.2}, \eqref{1.3} is unique.
\end{lemma}

\begin{proof} 
The proof is carried out by standard H\"olmgren's argument on the basis of Lemma~\ref{L2.1}.
\end{proof}

Now we present a number of auxiliary lemmas on solubility of the linear problems in non-smooth case.

Introduce a space $\widetilde{X}^{k,\alpha}(\Pi_T)$, that is different from $X^{k,\alpha}(\Pi_T)$ in that the condition $u\in C_w([0,T];H^{k,\alpha})$ is substituted by $u\in C([0,T];H^{k,\alpha})$, and let $\widetilde{X}^\alpha(\Pi_T)=\widetilde{X}^{0,\alpha}(\Pi_T)$. 

\begin{lemma}\label{L2.3}
Let $u_0\in L_2^\alpha$ for some $\alpha\geq 0$, $f\equiv f_0+\delta^{1/2}f_{1x}$, where $f_0\in L_1(0,T; L_2^\alpha)$, $f_1\in L_2(0,T;L_2^\alpha)$. Then there exists a (unique) generalized solution to each problem \eqref{2.1}, \eqref{1.2}, \eqref{1.3} $u(t,x,y)$ from the space $\widetilde{X}^\alpha(\Pi_T)$, and $\delta|Du|\in L_2(0,T;L_2^\alpha)$. Moreover, for any $t\in (0,T]$ uniformly with respect to $\delta$
\begin{multline}\label{2.3}
\|u\|_{\widetilde{X}^\alpha(\Pi_t)}+\delta^{1/2}\bigl\||Du|\bigr\|_{L_2(0,t;L_2^\alpha)} \\
\leq c(T) \left[\|u_0\|_{L_2^\alpha}+\|f_0\|_{L_1(0,t;L_2^\alpha)}+\|f_1\|_{L_2(0,t;L_2^\alpha)}\right],
\end{multline}
\begin{multline}\label{2.4}
\iint u^2(t,x,y)\rho(x)\,dxdy+\int_0^t\! \iint |Du|^2\cdot(\rho'+\delta\rho)\,dxdyd\tau \\
\leq \iint u_0^2\rho \,dxdy+c\int_0^t\! \iint u^2\rho\,dxdyd\tau
+2\int_0^t\!\iint f_0u\rho \,dxdyd\tau\\
-2\delta^{1/2}\int_0^t\! \iint f_1(u\rho)_x \,dxdyd\tau,
\end{multline}
where $\rho$ is an admissible weight function such that $\rho(x)\leq \text{const}\cdot (1+x_+)^{2\alpha}$ and the constant c depends on the properties of the function $\rho$.
\end{lemma}

\begin{proof}
It is sufficient to consider smooth solutions that were constructed, for example, in Lemma~\ref{L2.1} because of linearity of the problem.

Then multiplying equation \eqref{2.1} by $2u(t,x,y)\rho(x)$ and integrating over $\Sigma$ we obtain an equality
\begin{multline}\label{2.5}
\frac{d}{dt}\iint u^2\rho \,dxdy+\iint (3u_x^2+u_y^2)\rho'\,dxdy+2\delta\iint(u_x^2+u_y^2)\rho \,dxdy \\
-\iint(u^2\rho'''+\delta u^2\rho'')\,dxdy=2\iint f_0u\rho \,dxdy-2\delta^{1/2}\iint f_1(u\rho)_x \,dxdy, 
\end{multline}
which implies inequality \eqref{2.4} by the properties of admissible weight functions.

Next, let $\rho(x)\equiv 1+\rho_{0,1}(x-x_0)$ for any $x_0\in\mathbb R$ and then $\rho(x)\equiv \rho_{2\alpha,1}(x)$ if $\alpha>0$. Thus we obtain from \eqref{2.4} estimate \eqref{2.3}, which, in particular, allows us to prove Lemma in the case of non-smooth solutions.
\end{proof}

\begin{lemma}\label{L2.4}
Let $\delta = 0$, $u_0\rho_{\alpha,\beta}(x)\in L_2$ for some $\alpha>0$ and $\beta>0$, $f\equiv f_0+f_{1x}$, where $f_0\rho_{\alpha,\beta}(x)\in L_1(0,T;L_2)$, $f_1\rho_{\alpha+1/2,\beta}(x)\in L_2(\Pi_T)$. Then there exists a (unique) generalized solution to each of problems \eqref{2.1}, \eqref{1.2}, \eqref{1.3} such that $u\rho_{\alpha,\beta}(x)\in C([0,T];L_2)$, $|Du|\rho'_{\alpha,\beta}(x)\in L_2(\Pi_T)$ and inequality~\eqref{2.4}, where there is no multiplier $\delta^{1/2}$ in the last term in its right part, holds for any $t\in (0,T]$ and for $\rho\equiv\rho_{2\alpha,2\beta}$.
\end{lemma}

\begin{proof}
First of all, note that if $u\rho_{\alpha,\beta}\in C([0,T]; L_2)$ then $u$ belongs to the class $\bigl(C^\infty([0,T]; \EuScript S(\overline{\Sigma})\cap \EuScript S_{exp}(\overline{\Sigma}_-))\bigr)'$.

Next, the functions $u_0,f_0, f_1$ can be regularized such that $u_0\in L_2$, $f_0+f_{1x}\in L_1(0,T;L_2^\alpha)$ (for example, by substituting $u_0$ with the function $u_0\eta(x+1/h)$), so one can consider the solution $u\in \widetilde{X}^\alpha(\Pi_T)$. Equality \eqref{2.4} is valid for such solutions, where $f_0$ should be substituted by $f_0+f_{1x}$, and the term with $\delta^{1/2}f_1$ should be omitted. Also, it follows from \eqref{2.5} that $|Du|^2$ can be substituted  by $(3u_x^2+u_y^2)$ in the second summand of the left part in \eqref{2.4}.

Note that
\begin{equation}\label{2.6}
\frac{\rho^2_{2\alpha,2\beta}(x)}{\rho'_{2\alpha,2\beta}(x)}\leq c(\alpha,\beta)\rho_{2\alpha+1,\beta}(x), 
\end{equation}
and, therefore,
\begin{multline*}
2\iint f_{1x} u\rho_{2\alpha,2\beta} \,dxdy =-\iint f_1(u\rho_{2\alpha,2\beta})_x \,dxdy \leq \iint u_x^2\rho'_{2\alpha,2\beta}\,dxdy \\
+c\iint f_1^2\frac{\rho^2_{2\alpha,2\beta}}{\rho'_{2\alpha,2\beta}}\,dxdy +c\iint(u^2+f^2_1)\rho_{2\alpha,2\beta}\,dxdy \leq\iint u_x^2\rho'_{2\alpha,2\beta}\,dxdy\\
+c_1\iint f^2_1\rho_{2\alpha+1,2\beta}\,dxdy+ c\iint(u^2+f^2_1)\rho_{2\alpha,2\beta}\,dxdy.
\end{multline*}
\end{proof}

\begin{lemma}\label{L2.5}
Let $u_0\in H^{1,\alpha}$ for some $\alpha\geq 0$, $f\equiv f_0+\delta^{1/2}f_1$, where $f_0\in L_1(0,T;H^{1,\alpha})$, $f_1\in L_2(0,T;L_2^\alpha)$. Assume that $u_0, f_0$ satisfy the same assumptions as the ones for $u_0, f$ when $y=0, y=L$ in Theorem~\ref{T1.2}. Then there exists a (unique) generalized solution to each of problems \eqref{2.1}, \eqref{1.2}, \eqref{1.3} $u(t,x,y)$ in the space $\widetilde{X}^{1,\alpha}(\Pi_T)$, and $\delta|D^2u|\in L_2(0,T;L_2^\alpha)$ . Moreover, for any $t\in (0,T]$ uniformly with respect to $\delta$
\begin{multline}\label{2.7}
\|u\|_{\widetilde{X}^{1,\alpha}(\Pi_t)}+ \delta^{1/2}\bigl\||D^2u|\bigr\|_{L_2(0,t;L_2^\alpha)} \\
\leq c(T)\left[\|u_0\|_{H^{1,\alpha}}+\|f_0\|_{L_1(0,t;H^{1,\alpha})}+\|f_1\|_{L_2(0,t;L_2^\alpha)}\right],
\end{multline}
\begin{multline}\label{2.8}
\iint|Du(t,x,y)|^2\rho(x)\,dxdy
+\int_0^t\!\iint|D^2u|^2\cdot(\rho'+\delta\rho)\,dxdyd\tau \\
\leq \iint|Du_0|^2\rho \,dxdy+
c\int_0^t\!\iint |Du|^2\rho\,dxdyd\tau \\
+2\int_0^t\!\iint(f_{0x}u_x+f_{0y}u_y)\rho \,dxdyd\tau
-2\delta^{1/2}\int_0^t\iint f_1[(u_x\rho)_x+u_{yy}\rho]\,dxdyd\tau, 
\end{multline}
where $\rho(x)$ is the same function as in inequality \eqref{2.4}.
\end{lemma}

\begin{proof}
In the smooth case multiplying \eqref{2.1} by $-2\bigl(u_x(t,x,y)\rho(x)\bigr)_x
-2u_{yy}(t,x,y)\rho(x)$ and integrating over $\Sigma$ we obtain an equality
\begin{multline*}
\frac{d}{dt}\iint(u_x^2+u_y^2)\rho \,dxdy+ 
\iint(3u_{xx}^2+4u^2_{xy}+u^2_{yy})\rho \,dxdy \\
+2\delta\iint(u^2_{xx}+2u^2_{xy}+u^2_{yy})\rho \,dxdy 
-\iint(u^2_x+u^2_y)(\rho'''+\delta\rho'')\,dxdy \\
=2\iint(f_{0x}u_x+f_{0y}u_y)\rho \,dxdy-2\delta^{1/2}\iint f_1((u_x\rho)_x+u_{yy}\rho)\,dxdy. 
\end{multline*}
The rest part of the proof is similar to the end of the proof of Lemma~\ref{L2.3}.
\end{proof}

\begin{lemma}\label{L2.6}
Let the hypothesis of Lemma~\ref{L2.5} be satisfied for some $\delta>0$.  Consider the solution $u\in \widetilde{X}^{1,\alpha}(\Pi_T)$ constructed there such that $|D^2u|\in L_2(0,T;L_2^\alpha)$. Let $\rho(x)$ be the same function as in inequalities \eqref{2.4} and \eqref{2.8}, and, moreover, $\rho(x)\geq 1$. Let the function $g(u)\in C^2(\mathbb R)$, $g(0)=0$, be such that $|g'(u)|, |g''(u)|\leq \text{const}$ $\forall u\in \mathbb R$. Then for any $t\in (0,T]$ the following equality holds: 
\begin{multline}\label{2.10}
-2\iint g^*\bigl(u(t,x,y)\bigr)\rho(x)\,dxdy
+2\int_0^t\!\iint g'(u)u_x\cdot(u_{xx}+u_{yy})\rho \,dxdyd\tau \\
+2\int_0^t\!\iint g(u)\cdot (u_{xx}+u_{yy})\rho' \,dxdyd\tau 
-2\delta\int_0^t\!\iint g'(u)\cdot(u_x^2+u_y^2)\rho\,dxdyd\tau \\
-2\delta\int_0^t\!\iint g(u)u_x\rho'\,dxdyd\tau 
=-2\iint g^*(u_0)\rho \,dxdy-2\int_0^t\!\iint fg(u)\rho \,dxdyd\tau
\end{multline}
(recall that $g^*(u)\equiv\int_0^u g(\theta)\,d\theta$).
\end{lemma}

\begin{proof}. In the smooth case multiplying \eqref{2.1} by $-2g\bigl((u(t,x,y)\bigr)\rho(x)$ and integrating one instantly obtains equality \eqref{2.10}.

In order to obtain this equality in general case we use the passage to the limit. In this case, the presence of the terms of order higher than quadratic requires appropriate justification for this procedure. Note that $|g(u)|\leq c|u|$, $|g^*(u)|\leq cu^2$, $|g^*(u)-g^*(v)|\leq c(|u|+|v|)|u-v|$, $|g(u)-g(v)|\leq c|u-v|$, $|g'(u)-g'(v)|\leq c|u-v|$. 

Then, for example, if we denote by $u_h$ the corresponding smooth solution
\begin{multline*}
\int_0^t\!\iint |g'(u)-g'(u_h)|\cdot|u_x|\cdot(|u_{xx}|+|u_{yy}|)\rho \,dxdyd\tau \\
\leq c\int_0^t\!\iint |u-u_h|\cdot|u_x|\cdot(|u_{xx}|+|u_{yy}|)\rho^{3/2}\,dxdyd\tau \\
\leq c_1\sup_{\tau\in[0,t]}\Bigl(\iint|u-u_h|^4\rho^2\,dxdy\Bigr)^{1/4} \Bigl(\int_0^t\Bigl(\iint u_x^4\rho^2 \,dxdy\Bigr)^{1/2}\,d\tau\Bigr)^{1/2} \\
\times\Bigl(\int_0^t\!\iint |D^2u|^2\rho\, dxdyd\tau\Bigr)^{1/2}.
\end{multline*}
Since $u\rho^{1/2}, |Du|\rho^{1/2}\in C([0,T]; L_2)$, then $u\rho^{1/2}\in C([0,T];L_4)$ according to inequality \eqref{1.7} (for $q=4$, $\rho_1=\rho_2=\rho$) . Moreover, $|D^2u|\rho^{1/2}\in L_2(\Pi_T)$ and so $u_x\rho^{1/2}\in L_2(0,T; L_4)$. Thus, the passage to the limit is justified in this case. The other terms can be considered in a similar way.
\end{proof}

\begin{lemma}\label{L2.7}
Let $\delta=0$, $u_0\rho_{\alpha,\beta}(x)\in H^1$ for some $\alpha>0$ and $\beta>0$, $f\equiv f_0+f_1$, where $f_0\rho_{\alpha,\beta}(x)\in L_1(0,T;H^1)$, $f_1\rho_{\alpha+1/2,\beta}(x)\in L_2(\Pi_T)$. Assume that $u_0$ and $f_0$ satisfy the same assumptions as the ones for $u_0$, $f$ when $y=0, y=L$ in Theorem~\ref{T1.2}. Then there exists a (unique) generalized solution to each of problems \eqref{2.1}, \eqref{1.2}, \eqref{1.3} $u(t,x,y)$ such that  $u\rho_{\alpha,\beta}(x)\in C([0,T];H^1)$, $|D^2u|\rho'_{\alpha,\beta}(x)\in L_2(\Pi_T)$ and for any $t\in (0,T]$ inequality \eqref{2.8} holds for $\rho\equiv\rho_{2\alpha,2\beta}$, where there is no multiplier $\delta^{1/2}$ in the last term in the right part, and there should be a positive coefficient less than 1 before the second term in the left part. 
\end{lemma}

\begin{proof}
Regularize the functions $u_0$, $f_0$, $f_1$ such that $u_0$, $f$ satisfy the hypothesis of Lemma~\ref{L2.5}, where $f_0$ is substituted by $f_0+f_1$, and $f_1\equiv 0$ (we can also assume that when $y=0, y=L$ the assumptions on the function $f_0$ are also true for the function $f_1$). Let us consider the corresponding solution $u\in \widetilde{X}^{1,\alpha}(\Pi_T)$ and corresponding inequality \eqref{2.8}, where $f_0$ is substituted by $f_0+f_1$ and there is no term $\delta^{1/2}f_1$. 

Since 
\begin{multline*}
\Bigl|\int_0^t\!\iint(f_{1x}u_x+f_{1y}u_y)\rho_{2\alpha,2\beta}\,dxdyd\tau\Bigr| \\
=\Bigl|\int_0^t\!\iint f_1\bigl[(u_x\rho_{2\alpha,2\beta})_x
+u_{yy}\rho_{2\alpha,2\beta}\bigr] \,dxdyd\tau\Bigr|\\
\leq\Bigl(\int_0^t\!\iint(u_{xx}^2+u_{yy}^2+u_x^2)\rho'_{2\alpha,2\beta}
\,dxdyd\tau\Bigr)^{1/2}
\Bigl(\int_0^t\!\iint f_1^2\frac{\rho^2_{2\alpha,2\beta}}{\rho'_{2\alpha,2\beta}}\,dxdyd\tau\Bigr)^{1/2},
\end{multline*}
then according to inequality \eqref{2.6} one can finish the proof by the passage to the limit.
\end{proof}

\section{Existence of weak solutions}\label{S3}

Consider an equation of more general than \eqref{1.1} type:
\begin{equation}\label{3.1}
u_t+u_{xxx}+u_{xyy}-\delta(u_{xx}+u_{yy})+(g(u))_x=f(t,x,y).
\end{equation}

\begin{definition}\label{D3.1}
A function $u\in L_\infty(0,T;L_2)$ is called a weak solution to problem \eqref{3.1}, \eqref{1.2}, \eqref{1.3} if the function $g(u(t,x,y))\in L_1((0,T)\times(-r,r)\times(0,L))$ for any $r>0$ and for any function $\varphi\in C^\infty(\overline{\Pi}_T)$ such that $\varphi|_{t=T}=0$, $\varphi(t,x,y)=0$ when $|x|\geq r$ for some $r>0$ and $\varphi|_{y=0}=\varphi|_{y=L}=0$ in the case a, $\varphi_y|_{y=0}=\varphi_y|_{y=L}=0$ in the case b, $\varphi|_{y=0}=\varphi_y|_{y=L}=0$ in the case c, $\varphi|_{y=0}=\varphi|_{y=L}, \varphi_y|_{y=0}=\varphi_y|_{y=L}$ in the case d, the following equality holds:
\begin{multline}\label{3.2}
\iiint_{\Pi_T}\bigl[u(\varphi_t+\varphi_{xxx}+\varphi_{xyy}+\delta\varphi_{xx}
+\delta\varphi_{yy}) 
+g(u)\varphi_{x}+f\varphi\bigr]\,dxdyd\tau \\
+\iint_\Sigma u_0\varphi|_{t=0}\,dxdy=0.
\end{multline}
\end{definition}

\begin{remark}\label{R3.1}
It is easy to see that if $g\equiv 0$ and a function $u$ is a weak solution to problem \eqref{3.1}, \eqref{1.2}, \eqref{1.3}, then it is a generalized solution to problem \eqref{2.1}, \eqref{1.2}, \eqref{1.3} in the sense of Definition~\ref{D2.1}, though the class of test functions $\varphi$ in that definition is broader, than the one in Definition~\ref{D3.1}.
\end{remark}

\begin{remark}\label{R3.2}
If a weak solution to problem \eqref{3.1}, \eqref{1.2}, \eqref{1.3} $u\in L_\infty(0,T; H^1)$ and the function $g$ has the rate of growth not higher than polynomial when $|u|\to\infty$ (for example, it satisfies inequality \eqref{1.5} for some $b>0$), then according to \eqref{1.7} $g(u(t,x,y))\in L_\infty(0,T;L_2)$ and thus equality \eqref{3.2} also holds for the test functions $\varphi$ from Definition~\ref{D2.1}.
\end{remark}

First of all, we prove a lemma on solubility of problem \eqref{3.1}, \eqref{1.2}, \eqref{1.3} for spaces $L_2^\alpha$ in the  "regularized" case.

\begin{lemma}\label{L3.1}
Let $\delta>0$, $g\in C^1(\mathbb R)$ and $|g'(u)|\leq c\ \forall u\in\mathbb R$. Assume that $u_0\in L_2^\alpha$ for some $\alpha\geq 0$, $f\in L_1(0,T;L_2^\alpha)$. Then each of problems \eqref{3.1}, \eqref{1.2}, \eqref{1.3} has a unique solution $u\in C([0,T];L_2^\alpha)\cap L_2(0,T;H^{1,\alpha})$.
\end{lemma}

\begin{proof}
We apply the contraction principle. For $t_0\in(0,T]$ define a mapping $\Lambda$ on a set $Y^\alpha(\Pi_{t_0})=C([0,t_0];L_2^\alpha)\cap L_2(0,t_0;H^{1,\alpha})$ as follows: $u=\Lambda v\in Y^\alpha(\Pi_{t_0})$ is a solution to a linear problem
\begin{equation}\label{3.3}
u_t+u_{xxx}+u_{xyy}-\delta u_{xx}-\delta u_{yy}=f-(g(v))_x\
\end{equation}
in $\Pi_{t_0}$ with boundary conditions \eqref{1.2}, \eqref{1.3}.

Note that $|g(v)|\leq c|v|$ and, therefore,
$$
\|g(v)\|_{L_2(0,t_0;L_2^\alpha)}\leq c||v||_{L_2(0,t_0;L_2^\alpha)}<\infty.
$$

Thus, according to Lemma~\ref{L2.3} (where $f_1\equiv\delta^{-1/2}g(v)$) the mapping $\Lambda$ exists. Moreover, for functions $v,\widetilde{v}\in Y^\alpha(\Pi_{t_0})$
$$
\|g(v)-g(\widetilde{v})\|_{L_2(0,t_0;L_2^\alpha)}\leq c\|v-\widetilde{v}\|_{L_2(0,t_0;L_2^\alpha)}\leq ct_0^{1/2}\|v-\tilde{v}\|_{C([0,t_0];L_2^\alpha)}.
$$
As a result, according to inequality \eqref{2.3}
$$
\|\Lambda v-\Lambda\widetilde{v}\|_{Y^\alpha(\Pi_{t_0})}\leq 
c(T,\delta)t_0^{1/2}\|v-\widetilde{v}\|_{Y^\alpha(\Pi_{t_0})}.
$$
\end{proof}

Now we pass to the proof of Theorem~\ref{T1.1}.

\begin{proof}[Proof of Theorem~\ref{T1.1}] 
For $h\in (0,1]$ consider a set of initial-boundary value problems in $\Pi_T$
\begin{equation}\label{3.4}
u_t+u_{xxx}+u_{xyy}-hu_{xx}-hu_{yy}+(g_h(u))_x=f(t,x,y)
\end{equation}
with boundary conditions \eqref{1.2}, \eqref{1.3}, where
\begin{equation}\label{3.5}
g_h(u)\equiv\int_0^u\Bigl[g'(\theta)\eta(2-h|\theta|)+g'\bigl(\frac{2\sgn\theta}{h}\bigr)\eta(h|\theta|-1)\Bigr]\,d\theta.
\end{equation}
Note that $g_h(u)\equiv g(u)$ if $|u|\leq 1/h$ and $|g'_h(u)|\leq c(h^{-1})\ \forall u\in\mathbb R$. 

According to Lemma~\ref{L3.1} there exists a unique solution to this problem $u_h\in C([0,T];L_2^\alpha)\cap L_2(0,T;H^{1,\alpha})$. Moreover, $|g'_h(u)|\leq c(1+|u|^b)$ uniformly with respect to $h$.

Next, establish appropriate estimates for functions $u_h$ uniform with respect to $h$.

Write down corresponding inequality \eqref{2.4} for functions $u_h$ (we omit the index $h$ in intermediate steps for simplicity):
\begin{multline}\label{3.6}
\iint u^2\rho \,dxdy
+\int_0^t\!\iint |Du|^2\cdot(\rho'+h\rho)\,dxdyd\tau \leq \iint u_0^2\rho \,dxdy \\
+c\int_0^t\!\iint u^2\rho \,dxdyd\tau+2\int_0^t\!\iint fu\rho \,dxdyd\tau 
-2\int_0^t\!\iint g'(u)u_x u\rho \,dxdyd\tau.
\end{multline}

Firstly let $\rho\equiv 1$, then since
\begin{equation}\label{3.7}
g'(u)u_xu=\Bigl(\int_0^u g'(\theta)\theta \,d\theta\Bigr)_x \equiv \bigl(g'(u)u\bigr)^*_x
\end{equation}
we have that $\iint g'(u)u_x u\,dxdy=0$ and inequality \eqref{3.6} yields that 
\begin{equation}\label{3.8}
\|u_h\|_{C([0,T];L_2)}+h^{1/2}\|u_h\|_{L_2(0,T;H^1)}\leq c
\end{equation}
uniformly with respect to $h$.

Now let $\rho(x)$ be an admissible weight function such that $\rho'(x)$ is also an admissible weight function, and $\rho'(x)\leq c\rho(x)$. 
Then according to \eqref{3.7}
$$
\Bigl|\iint g'(u)u_xu\rho \,dxdy\Bigr|
=\Bigl|\iint(g'(u)u)^*\rho'\,dxdy\Bigr|\leq c\iint(u^2+|u|^{b+2})\rho'\,dxdy.
$$
Apply interpolating inequality \eqref{1.7} for $k=1$, $m=0$, $\rho_1=\rho_2\equiv\rho'$:
\begin{multline}\label{3.9}
\iint |u|^{b+2}\rho'\,dxdy \leq \Bigl(\iint u^2\,dxdy\Bigr)^{b/2} \Bigl(\iint|u|^{4/(2-b)}(\rho')^{2/(2-b)}\,dxdy\Bigr)^{1-b/2} \\
\leq c\Bigl(\iint u^2\,dxdy\Bigr)^{b/2}\Bigl[\Bigl(\iint |Du|^2\rho'dxdy\Bigr)^{b/2} \Bigl(\iint u^2\rho \,dxdy\Bigr)^{1-b/2}\\
+\iint u^2\rho \,dxdy\Bigr].
\end{multline}
Since $b<2$ and the norm of the solution in the space $L_2$ is already  estimated in \eqref{3.8}, it follows from \eqref{3.6} that
\begin{multline*}
\iint u^2\rho \,dxdy+\frac{1}{2}\int_0^t\!\iint |Du|^2\rho' \,dxdy+h\int_0^t\!\iint|Du|^2\rho dxdyd\tau \\
\leq\iint u_0^2\rho \,dxdy+c\int_0^t\! \iint u^2\rho \,dxdyd\tau+2\int_0^t\!\iint fu\rho dxdyd\tau,
\end{multline*}
and choosing the functions $\rho$ in the same way as in the proof of Lemma~\ref{L2.3} we obtain that uniformly with respect to $h$
\begin{equation}\label{3.10}
\|u_h\|_{X^\alpha(\Pi_T)}+h^{1/2}\|u_h\|_{L_2(0,T;H^{1,\alpha})}\leq c.
\end{equation}

In particular, $\|u_h\|_{L_2(0,T;H^1(Q_n))}\leq c(n)$ for any rectangle $Q_n=(-n,n)\times (0,L)$, and thus
\begin{equation}\label{3.11}
\|u_h^2\|_{L_2((0,T)\times Q_n)}\leq c(n).
\end{equation}
Since $|g_h(u)|\leq c(|u|+|u|^3)$ we have that $\|g_h(u_h)\|_{L_2(0,T;L_1(Q_n))}\leq c(n)$. Using the well-known embedding $L_1(\Omega)\subset H^{-2}(Q_n)$ for domains $\Omega\subset \mathbb R^2$ we first derive that $\|g_h(u_h)\|_{L_2(0,T;H^{-2}(Q_n))}\leq c(n)$, and then according to equation \eqref{3.1} itself that uniformly with respect to $h$
$$
\|u_{ht}\|_{L_1(0,T;H^{-3}(Q_n))}\leq c(n).
$$
Applying the compactness embedding theorem of evolutionary spaces from \cite{S2} we obtain that the set $\{u_h\}$ is precompact in $L_2((0,T)\times Q_n)$ for all $n$.

Now show that if $u_h\to u$ in $L_2((0,T)\times Q_n)$ for some sequence $h\to 0$, then $g_h(u_h)\to g(u)$ in $L_1((0,T)\times Q_n)$.
Indeed,
\begin{multline*}
|g_h(u_h)-g(u)|\leq |g_h(u_h)-g_h(u)|+|g_h(u)-g(u)| \\
\leq c(1+u_h^2+u^2)|u_h-u|+|g_h(u)-g(u)|
\end{multline*}
and then we can use estimate \eqref{3.11}.

As a result, the required solution is constructed in a standard way as the limit of the solutions $u_h$ when $h\to 0$.
\end{proof}

We now proceed to solutions in spaces $H^{1,\alpha}$ and firstly estimate a lemma analogous to Lemma~\ref{L3.1}.

\begin{lemma}\label{L3.2}
Let $\delta>0$, $g\in C^2(\mathbb R)$ and $|g'(u)|,|g''(u)|\leq c \ \forall u\in\mathbb R$. Assume that $u_0\in H^{1,\alpha}$ for some $\alpha\geq 0$, $f\in L_1(0,T;H^{1,\alpha})$ and the same assumptions as in Theorem~\ref{T1.2} hold for the functions $u_0$, $f$ when $y=0$, $y=L$. Then each of problems \eqref{3.1}, \eqref{1.2}, \eqref{1.3} has a unique solution $u\in C([0,T];H^{1,\alpha})\cap L_2(0,T;H^{2,\alpha})$.
\end{lemma}

\begin{proof}
Introduce for $t_0\in (0,T]$ a space $Y^{1,\alpha}(\Pi_{t_0})=C([0,t_0];H^{1,\alpha})\cap L_2(0,t_0;H^{2,\alpha})$ and define a mapping $\Lambda$ on it in the same way as in the proof of Lemma~\ref{L3.1} (with the substitution of $Y^\alpha$ by $Y^{1,\alpha}$). Since $|g'(v)v_x|\leq c|v_x|$ then
$$
||g'(v)v_x||_{L_2(0,t_0;L_2^\alpha)}\leq ct_0^{1/2}||v||_{C([0,t_0];H^{1,\alpha})}
$$
and according to Lemma~\ref{L2.5} (where $f_1\equiv\delta^{-1/2}g'(v)v_x$) such a mapping $\Lambda$ exists. Moreover, according to \eqref{2.7}
\begin{equation}\label{3.12}
\|\Lambda v\|_{Y^{1,\alpha}(\Pi_{t_0})}\leq c\left(\|u_0\|_{H^{1,\alpha}}+ 
t_0^{1/2}\|v\|_{Y^{1,\alpha}(\Pi_{t_0})}+1\right).
\end{equation}
Besides that, since $|g'(v)v_x-g'(\widetilde{v})\widetilde{v}_x|\leq c|v_x|\cdot |v-\widetilde{v}|+c|v_x-\tilde{v}_x|$ 
\begin{multline}\label{3.13}
\|\Lambda v-\Lambda\widetilde{v}\|_{Y^{1,\alpha}(\Pi_{t_0})}\leq c||g'(v)v_x-g'(\widetilde{v})\widetilde{v}_x||_{L_2(0,t_0;L_2^\alpha)} \\ 
\leq c\Bigl[\sup_{t\in [0,t_0]}||v_x||_{L_2^\alpha}\bigl\|\sup_{(x,y)\in\Sigma}
|v-\widetilde{v}|\bigr\|_{L_2(0,t_0)}+t_0^{1/2}\|v_x-\widetilde{v}_x\|_{C([0,t_0];L_2^\alpha)}\Bigr] \\
\leq c_1\Bigl[\sup_{t\in[0,t_0]}\|v\|_{H^{1,\alpha}}\cdot 
\sup_{t\in[0,t_0]}\|v-\widetilde{v}\|_{L_2}^{1/2}\cdot t_0^{1/4}\cdot
\|v-\widetilde{v}\|_{L_2(0,t_0;H^2)}^{1/2} \\
+t_0^{1/2}\|v-\widetilde{v}\|_{Y^{1,\alpha}(\Pi_{t_0})}\Bigr] \leq c_2 t_0^{1/4} (1+\sup_{t\in [0,t_0]}\|v\|_{H^{1,\alpha}})\|v-\widetilde{v}\|_{Y^{1,\alpha}(\Pi_{t_0})},
\end{multline}
where we used inequality \eqref{1.7} for $q=+\infty$.

Now we estimate the following a priori estimate: if $u\in Y^{1,\alpha}(\Pi_{T'})$ is a solution to the considered problem for some $T'\in (0,T]$ then
\begin{equation}\label{3.14}
\|u\|_{C([0,T'];H^{1,\alpha})}\leq c(\|u_0\|_{H^{1,\alpha}}, 
\|f\|_{L_1(0,T';H^{1,\alpha})}).
\end{equation} 

In fact, using inequalities \eqref{2.4}, where $f_0\equiv f-g'(u)u_x$, $f_1\equiv 0$ and \eqref{2.8} where $f_0\equiv f$, $f_1\equiv -\delta^{-1/2}g'(u)u_x$ for the function $\rho(x)\equiv 1+\rho_{2\alpha,1}(x)$ we obtain that
\begin{multline*}
\iint(u^2+|Du|^2)\rho \,dxdy+\delta\int_0^t\iint|D^2u|^2\rho \,dxdyd\tau 
\leq\iint(u_0^2+|Du_0|^2)\rho \,dxdy \\
+c\int_0^t\iint(u^2+|Du|^2)\rho \,dxdyd\tau +2\int_0^t\iint(fu+f_xu_x+f_yu_y)\rho \,dxdyd\tau \\
+c\int_0^t\Bigl(\iint(u^2_{xx}+u^2_{yy})\rho \,dxdy\Bigr)^{1/2} \Bigl(\iint u_x^2\rho \,dxdy\Bigr)^{1/2}d\tau,
\end{multline*}
which implies estimate \eqref{3.14}.

Inequalities \eqref{3.12} and \eqref{3.13} allow us to construct a solution to the considered problem locally in time by the contraction while estimate \eqref{3.14} enables us to extend it for the whole time segment $[0,T]$.
\end{proof}

At the end of this section we present the proof of the part of Theorem~\ref{T1.2} concerning existence of solutions.

\begin{proof}[Proof of Theorem~\ref{T1.2}, existence] 
As in the proof of Theorem~\ref{T1.1} consider the set of "regularized" problems \eqref{3.1}, \eqref{1.2}, \eqref{1.3}. Note that according to formula \eqref{3.5} $g'_h(u)=g'(\sgn u/h)$ when $|u|\geq 2/h$ and thus $|g''_h(u)|\leq c(h^{-1})\ \forall u\in\mathbb R$. 

Consider solutions to these problems $u_h\in C([0,T]; H^{1,\alpha})\cap L_2(0,T;H^{2,\alpha})$ and establish for them estimates uniform with respect to $h$ (note that estimate \eqref{3.10} remains valid). Write down corresponding inequalities \eqref{2.8}, \eqref{2.10} for these functions (where $f_0\equiv f, f_1\equiv -h^{-1/2}g'(u)u_x)$ and summarize them (index $h$ is again omitted):
\begin{multline}\label{3.15}
\iint(|Du|^2-2g^*(u))\rho \,dxdy+\int_0^t\!\iint|D^2u|^2\cdot(\rho'+h\rho) \,dxdyd\tau \\
\leq \iint(|Du_0|^2-2g^*(u_0))\rho \,dxdy+c\int_0^t\!\iint|Du|^2\rho \,dxdyd\tau \\
+2\int_0^t\!\iint (f_xu_x+f_yu_y)\rho \,dxdyd\tau 
-2\int_0^t\!\iint fg(u)\rho \,dxdyd\tau \\ 
+2\int_0^t\!\iint g'(u)u^2_x\rho' \,dxdyd\tau
-2\int_0^t\!\iint g(u)(u_{xx} +u_{yy})\rho' \,dxdyd\tau \\
+2h\int_0^t\!\iint g'(u)(u_x^2+u_y^2)\rho \,dxdyd\tau 
+2h\int_0^t\!\iint g(u)u_x\rho' \,dxdyd\tau \\ 
+2\int_0^t\!\iint g'(u)g(u)u_x\rho \,dxdyd\tau.
\end{multline}
Note that $|g^*(u)|\leq c(u^2+|u|^{b+2})$ and so $\iint g^*(u)\rho \,dxdy$ can be estimated similarly to \eqref{3.9} (with the replacement of $\rho'$ by $\rho$). 

Next, $|fg(u)|\leq c|f|(|u|+|u|^3)$ and so 
\begin{multline*}
\iint|fu|^3\rho \,dxdy\leq c\Bigl(\iint f^4\rho^2 \,dxdy\Bigr)^{1/4} 
\Bigl(\iint u^4\rho^2dxdy\Bigr)^{3/4} \\
\leq c_1\Bigl(\iint(|Df|^2+f^2)\rho \,dxdy\Bigr)^{1/2} 
\Bigl[\Bigl(\iint |Du|^2\rho dxdy\Bigr)^{3/4}\Bigl(\iint u^2\rho dxdy)^{3/4} \\
+\Bigl(\iint u^2\rho dxdy\Bigr)^{3/2}\Bigr],
\end{multline*}
then using inequality \eqref{1.7} we find that
\begin{multline*}
h\iint |u|^b|Du|^2\rho \,dxdy \leq h\Bigl(\iint |u|^{2b}\rho^b \,dxdy\Bigr)^{1/2}
\Bigl(\iint |Du|^4\rho^2 \,dxdy\Bigr)^{1/2} \\
\leq ch \Bigl(\iint |D^2u|^2\rho \,dxdy\Bigr)^{(b+2)/4} 
\Bigl(\iint u^2\rho dxdy\Bigr)^{(b+2)/4}+ch \Bigl(\iint u^2\rho dxdy\Bigr)^{(b+2)/2}.
\end{multline*}
Thus, setting firstly in \eqref{3.15} $\rho\equiv 1$ we find that
\begin{equation}\label{3.16}
\|u_h\|_{C([0,T];H^1)}+h^{1/2}\bigl\||D^2u|\bigr\|_{L_2(\Pi_T)}\leq c.
\end{equation}
Now let $\rho(x)$ be an admissible weight function such that its derivative $\rho'$ is also an admissible weight function and $\rho'(x)\leq c\rho(x)$. Estimate the terms in the right part of \eqref{3.15}:
\begin{multline*}
\iint u^2u^2_x\rho' \,dxdy \leq \Bigl(\iint u^4 \,dxdy\Bigr)^{1/2}
\Bigl(\iint u^4_x(\rho')^2\, dxdy\Bigr)^{1/2} \\
\leq c\sup_{t\in [0,T]}\|u\|^2_{H^1} 
\Bigl[\Bigl(\iint |Du_x|^2\rho' \,dxdy\Bigr)^{1/2} 
\Bigl(\iint u^2_x\rho \,dxdy\Bigr)^{1/2}+\iint u^2_x\rho \,dxdy\Bigr],
\end{multline*}
\begin{multline*}
\iint |u|^3(|u_{xx}|+|u_{yy}|)\rho' \,dxdy \\ 
\leq c\Bigl(\iint u^8\,dxdy\Bigr)^{1/4} \Bigl(\iint u^4\rho^2 \,dxdy\Bigr)^{1/4}
\Bigl(\iint |D^2u|^2\rho'dxdy\Bigr)^{1/2} \\ 
\leq c_1\sup_{t\in[0,T]}\|u\|^2_{H^1} 
\Bigl(\iint(|Du|^2+u^2)\rho \,dxdy\Bigr)^{1/2}
\Bigl(\iint |D^2u|^2\rho'dxdy\Bigr)^{1/2},
\end{multline*}
$\iint |u|^3|u_x|\rho'\, dxdy$ is estimated similarly and, finally, 
$$
\Bigl|\iint g'(u)g(u)u_x\rho \,dxdy\Bigr| =\Bigl|\iint(g'(u)g(u))^*\rho'\,dxdy\Bigr| \leq c\iint(u^2+u^6)\rho \,dxdy,
$$
where the last integral also similarly can be estimated by virtue of \eqref{1.7}.

Thus, taking into account previously established estimates \eqref{3.10}, \eqref{3.16} and choosing the function $\rho$ in the same way as in the proof of Theorem~\ref{T1.1}, we find that
\begin{equation}\label{3.17}
\|u_h\|_{X^{1,\alpha}(\Pi_T)}+h^{1/2}\|u_h\|_{L_2(0,T;H^{2,\alpha})}\leq c.
\end{equation}

The end of the proof is exactly the same as for Theorem~\ref{T1.1}.
\end{proof}

\section{Continuous dependence of weak solutions}\label{S4}

First of all, we present a theorem from which the result of Theorem~\ref{T1.2} on uniqueness of weak solutions follows.

\begin{theorem}\label{T4.1}
Let $g\in C^2(\mathbb R)$ and inequality \eqref{1.6} is valid for certain $b\geq 1$. Let $u_0,\widetilde{u}_0\in H^{1,\alpha}$, $f,\widetilde{f}\in L_1(0,T;H^{1,\alpha})$ for some $\alpha\geq 1/2$, $u,\widetilde{u}$ be weak solutions to corresponding problems \eqref{1.1}--\eqref{1.3} from the class $X^{1,\alpha}(\Pi_T)$. Then for any $\beta>0$
\begin{multline}\label{4.1}
\|(u-\widetilde{u})\rho_{\alpha,\beta}(x)||_{L_\infty(0,T;L_2)}+\bigl\| |D(u-\widetilde{u})|\rho_{\alpha-1/2,\beta}(x)\bigr\|_{L_2(\Pi_T)} \\
\leq c\left(\|(u_0-\widetilde{u}_0)\rho_{\alpha,\beta}(x)\|_{L_2} +
\|(f-\widetilde{f})\rho_{\alpha,\beta}(x)\|_{L_1(0,T;L_2)}\right),
\end{multline}
where the constant $c$ depends on the norms of the functions $u,\widetilde{u}$ in the space $L_\infty(0,T;H^{1,1/2})$.
\end{theorem}

\begin{proof}
Let
\begin{equation}\label{4.2}
g_1(u)\equiv g(u)-g'(0)u.
\end{equation}
Then $g'(u)u_x=g'(0)u_x+(g_1(u))_x$, where $|g_1(u)|\leq c(u^2+|u|^{b+1})$. Note that for $q\geq 2$
\begin{multline*}
\iint |u|^{2q}\rho_{2\alpha+1,2\beta}\,dxdy 
\leq c\iint|u|^{2q}\cdot(1+x_+)^{2\alpha q}\,dxdy \\ 
\leq c_1\Bigl(\iint(u^2+|Du|^2)\cdot(1+x_+)^{2\alpha}\,dxdy\Bigr)^q,
\end{multline*}
that is $g_1(u)\rho_{\alpha+1/2,\beta}\in L_\infty(0,T;L_2)$. Besides that, it is obvious that $u_x\rho_{\alpha,\beta}\in L_\infty(0,T;L_2)$.

Denote $v\equiv u-\tilde{u}$, then the function $v$ is a solution to an equation
\begin{equation}\label{4.3}
v_t+v_{xxx}+v_{xyy}=\bigl(f-\widetilde{f})-(g(u)-g(\widetilde{u})\bigr)_x,
\end{equation}
satisfies an initial value condition
\begin{equation}\label{4.4}
v|_{t=0}=u_0-\tilde{u}_0,
\end{equation}
and corresponding boundary value conditions \eqref{1.3}.

The hypothesis of Lemma~\ref{L2.4} are true for this problem, therefore,
\begin{multline}\label{4.5}
\iint v^2\rho_{2\alpha,2\beta}\,dxdy
+\int_0^t\!\iint|Dv|^2\rho'_{2\alpha,2\beta}\,dxdyd\tau \leq \iint(u_0-\widetilde{u}_0)^2\rho_{2\alpha,2\beta}\,dxdy \\
+c\int_0^t\!\iint v^2\rho_{2\alpha,2\beta}\,dxdyd\tau 
+2\int_0^t\!\iint(f-\widetilde{f})v\rho_{2\alpha,2\beta}\,dxdyd\tau \\
-2g'(0)\int_0^t\!\iint v_xv\rho_{2\alpha,2\beta}\,dxdyd\tau
-2\int_0^t\!\iint\bigl(g_1(u)-g_1(\widetilde{u})\bigr)_xv\rho_{2\alpha,2\beta}\,dxdyd\tau. 
\end{multline}
It is easy to see that
\begin{multline*}
\Bigl|\iint\bigl(g_1(u)-g_1(\widetilde{u})\bigr)_xv\rho_{2\alpha,2\beta}\,dxdy\Bigr|  \\ 
\leq c\iint\bigl(1+|u|^{b-1}+|\widetilde{u}|^{b-1}\bigr) 
\bigl(|u_x|+|\tilde{u}_x|+|u|+|\widetilde{u}|\bigr)v^2\rho_{2\alpha,2\beta}\,dxdy.
\end{multline*}
Here
\begin{multline*}
\iint |u_x|v^2\rho_{2\alpha,2\beta}\,dxdy  \leq 
\Bigl(\iint u^2_x\frac{\rho_{2\alpha,2\beta}}{\rho'_{2\alpha,2\beta}}\,dxdy\Bigr)^{1/2}
\Bigl(\iint v^4\rho'_{2\alpha,2\beta}\rho_{2\alpha,2\beta}\,dxdy\Bigr)^{1/2} \\ 
\leq \Bigl(\iint(1+x_+)u_x^2\,dxdy\Bigr)^{1/2} 
\Bigl[\Bigl(\iint|Dv|^2\rho'_{2\alpha,2\beta}\,dxdy\Bigr)^{1/2}
\Bigl(\iint v^2\rho_{2\alpha,2\beta}\,dxdy\Bigr)^{1/2} \\ 
+\iint v^2\rho_{2\alpha,2\beta}\,dxdy\Bigr].
\end{multline*}

Further without loss of generality we assume that $b\geq2$. Denote $q=b-1$ then for any $p>4$
\begin{multline*}
\iint |u|^q |u_x|v^2\rho_{2\alpha,2\beta}\,dxdy   
\leq \Bigl(\iint|u|^{pq/(p-2)}|u_x|^{p/(p-2)}\frac{\rho_{2\alpha,2\beta}}{\rho'_{2\alpha,2\beta}} \,dxdy\Bigr)^{1-2/p} \\ \times
\Bigl(\iint |v|^p(\rho'_{2\alpha,2\beta})^{p/2-1}\rho_{2\alpha,2\beta}\,dxdy\Bigr)^{2/p},
\end{multline*}
where the first multiplier can be estimated as follows:
\begin{multline*}
c\Bigl(\iint |u|^{pq/(p-2)}(1+x_+)^{pq/(2(p-2))} \cdot
|u_x|^{p/(p-2)}(1+x_+)^{p/(2(p-2))}\,dxdy\Bigr)^{1-2/p} \\
\leq c_1\Bigl(\iint|u|^{2pq/(p-4)}(1+x_+)^{pq/(p-4)}\,dxdy\Bigr)^{(p-4)/(2p)}
\Bigl(\iint u^2_x\cdot(1+x_+)\,dxdy\Bigr)^{1/2}
\end{multline*}
and the second one in such a way:
$$
c\Bigl(\iint|Dv|^2\rho'_{2\alpha,2\beta}\,dxdy\Bigr)^{1-2/p}
\Bigl(\iint v^2\rho_{2\alpha,2\beta}\,dxdy\Bigr)^{2/p}
+c\iint v^2\rho_{2\alpha,2\beta}\,dxdy.
$$
The other terms can be estimated similarly and then \eqref{4.5} implies the required assertion.
\end{proof}

Finally we set a result on continuous dependence of weak solutions in more smooth spaces.

\begin{theorem}\label{T4.2}
Let the hypothesis of Theorem~\ref{T4.1} be true for $\alpha\geq3/4$ and also let the hypothesis of Theorem~\ref{T1.2} be true for the functions $u_0$, $\widetilde{u}_0$, $f$,$\widetilde{f}$ when $y=0$ and $y=L$. Then for any $\beta>0$
\begin{multline}\label{4.6}
\bigl\| |D(u-\widetilde{u})|\rho_{\alpha,\beta}(x)\bigr\|_{L_\infty(0,T;L_2)}+
\bigl\| |D^2(u-\widetilde{u})|\rho_{\alpha-1/2,\beta}\bigr\|_{L_2(\Pi_T)} \\
\leq c\left(\|u_0-\widetilde{u}_0\|_{H^{1,\alpha}}+
\|f-\widetilde{f}\|_{L_1(0,T;H^{1,\alpha})}\right),
\end{multline}
where the constant $c$ depends on the norms of the functions $u$ and $\widetilde{u}$ in $X^{1,\alpha}(\Pi_T)$.
\end{theorem}

\begin{proof}
Let $v\equiv u-\widetilde{u}$, $w(t,x,y)\equiv v(t,x+g'(0)t,y)$. Then the function $w$ satisfies an equation
\begin{equation}\label{4.7}
w_t+w_{xxx}+w_{xyy}=(f-\widetilde{f})-(g'_1(u)u_x-g'_1(\widetilde{u})\widetilde{u}_x),
\end{equation}
where the function $g_1$ is given by formula \eqref{4.2}, and all the terms in the right part of \eqref{4.7} are considered in the point $(t,x+g'(0)t,y)$, satisfies initial value condition \eqref{4.4} and corresponding boundary value conditions \eqref{1.3}.

Note that $|g'_1(u)u_x|\leq c(|u|+|u|^b)|u_x|$ and for $q\geq 1$
\begin{multline*}
\iint |u|^{2q}u_x^2\rho_{2\alpha+1,2\beta} \,dxdy 
\leq c\iint |u|^{2q}(1+x_+)^{3/2}\cdot u_x^2\rho_{2\alpha-1/2,2\beta}\,dxdy \\
\leq c_1\Bigl(\iint |u|^{4q}(1+x_+)^{4q\alpha}\,dxdy\Bigr)^{1/2}
\Bigl(\iint u_x^4\rho_{2\alpha-1,2\beta}^2\,dxdy\Bigr)^{1/2} \\
\leq c_2\Bigl(\iint(u^2+|Du|^2)\cdot(1+x_+)^{2\alpha}\,dxdy\Bigr)^q 
\Bigl[\Bigl(\iint |Du_x|^2\rho_{2\alpha-1,2\beta}\,dxdy)^{1/2} \\ \times
\Bigl(\iint u_x^2\rho_{2\alpha,2\beta}\,dxdy\Bigr)^{1/2} 
+\iint u^2_x\rho_{2\alpha,2\beta}\,dxdy\Bigr]
,
\end{multline*}
that is $g'_1(u)u_x\rho_{\alpha+1/2,\beta}\in L_2(\Pi_T)$. Thus, the hypothesis of Lemma~\ref{L2.7} are true for the considered problem and, therefore,
\begin{multline}\label{4.8}
\iint |Dw|^2\rho_{2\alpha,2\beta}\,dxdy+
\frac{1}{2}\int_0^t\! \iint |D^2w|^2\rho'_{2\alpha,2\beta}\, dxdyd\tau \\
\leq \iint |Du_0|^2\rho_{2\alpha,2\beta}\,dxdy+
c\int_0^t\!\iint |Dw|^2\rho_{2\alpha,2\beta}\,dxdyd\tau \\
+2\int_0^t\! \iint\bigl[(f-\widetilde{f})_x w_x+(f-\widetilde{f})_y w_y\bigr]\rho_{2\alpha,2\beta}\,dxdyd\tau \\
+2\int_0^t\!\iint \bigl(g'_1(u)u_x-g'_1(\widetilde{u})\widetilde{u}_x\bigr) \bigl[(w_x\rho_{2\alpha,2\beta})_x+w_{yy}\rho_{2\alpha,2\beta}\bigr] \,dxdyd\tau.
\end{multline}
The last integral in the right part of \eqref{4.7} is not greater than
\begin{multline*}
\varepsilon\int_0^t\!\iint (w_{xx}^2+w_{yy}^2+w_{x}^2)\rho'_{2\alpha,2\beta}\,dxdyd\tau \\ +c(\varepsilon)\int_0^t\!\iint\bigl(1+|u|^{2(b-1)}+
|\widetilde{u}|^{2(b-1)}\bigr)\ \bigl(u_x^2w^2+\widetilde{u}^2w_x^2\bigr)\rho_{2\alpha+1,2\beta}\,dxdyd\tau,
\end{multline*}
where $\varepsilon>0$ can be chosen arbitrarily small.
Here
\begin{multline*}
\iint u_x^2w^2\rho_{2\alpha+1,2\beta}\,dxdy \\
\leq c\Bigl(\iint u_x^4\rho_{2\alpha-1,\beta/2}\rho_{2\alpha,\beta/2}\,dxdy\Bigr)^{1/2}
\Bigl(\iint w^4\rho_{2\alpha,2\beta}\rho_{2\alpha,\beta}\,dxdy\Bigr)^{1/2} \\
\leq c_1\Bigl[\Bigr(\iint|Du_x|^2\rho_{2\alpha-1,\beta/2}\,dxdy\Bigr)^{1/2}
\Bigl(\iint u^2_x\rho_{2\alpha,\beta/2}\,dxdy\Bigr)^{1/2} \\
+\iint u^2_x\rho_{2\alpha,\beta/2}\,dxdy\Bigr] \\ 
\times\Bigl[\Bigl(\iint|Dw|^2\rho_{2\alpha,2\beta}\,dxdy\Bigr)^{1/2}
\Bigl(\iint w^2\rho_{2\alpha,\beta}\,dxdy\Bigr)^{1/2}
+\iint w^2\rho_{2\alpha,\beta}\,dxdy\Bigr] \\
\leq c_2\iint|Dw|^2\rho_{2\alpha,2\beta}\,dxdy+c_2\Bigl[\iint |Du_x|^2\rho_{2\alpha-1,\beta/2}\,dxdy+1\Bigr]\iint w^2\rho_{2\alpha,\beta}\,dxdy,
\end{multline*}
where the first multiplier in the last term belongs to the space $L_1(0,T)$ and the second one can be estimated uniformly with respect to $t$ according to inequality~\eqref{4.1}.
Similarly, if $q\geq1$ then
\begin{multline*}
\iint|u|^{2q}u_x^2w^2\rho_{2\alpha+1,2\beta}\,dxdy \\
\leq c\Bigl(\iint |u|^{8q}\,dxdy\Bigr)^{1/4}
\Bigl(\iint u_x^4\rho_{2\alpha-1,\beta/4}\rho_{2\alpha,\beta/4}\,dxdy\Bigr)^{1/2}  \\ 
\times \Bigl(\iint w^8\rho^3_{2\alpha,2\beta}\rho_{2\alpha,\beta}\,dxdy\Bigr)^{1/4} \\
\leq c_1\Bigl[\Bigl(\iint|Du_x|^2\rho_{2\alpha-1,\beta/4}\,dxdy\Bigr)^{1/2}
\Bigl(\iint u^2_x\rho_{2\alpha,\beta/4}\,dxdy\Bigr)^{1/2} \\
+\iint u^2_x\rho_{2\alpha,\beta/4}\,dxdy\Bigr] \\
\times \Bigl[\Bigl(\iint |Dw|^2\rho_{2\alpha,2\beta}\,dxdy\Bigr)^{3/4}
\Bigl(\iint w^2\rho_{2\alpha,\beta}\,dxdy\Bigr)^{1/4}
+\iint w^2\rho_{2\alpha,\beta}\,dxdy\Bigr] \\
\leq c_2\Bigl[\iint|Du_x|^2\rho_{2\alpha-1,\beta/4}\,dxdy+1\Bigr]\cdot
\Bigl[\iint|Dw|^2\rho_{2\alpha,2\beta}\,dxdy+\iint w^2\rho_{2\alpha,\beta}\,dxdy\Bigr],
\end{multline*}
where the first multiplier in the right part belongs to $L_1(0,T)$.
Finally, for $q\geq1$
\begin{multline*}
\iint |\tilde{u}|^{2q}w^2_x\rho_{2\alpha+1,2\beta} \,dxdy \\
\leq c\Bigl(\iint|\tilde{u}|^{4q}\cdot(1+x_+)^{4\alpha q}\,dxdy\Bigr)^{1/2}
\Bigl(\iint w^4_x\rho_{2\alpha-1,2\beta}\rho_{2\alpha,2\beta}\,dxdy\Bigr)^{1/2} \\
\leq c_1\Bigl(\iint|Dw_x|^2\rho'_{2\alpha,2\beta}\,dxdy\Bigr)^{1/2}
\Bigl(\iint w_x^2\rho_{2\alpha,2\beta}\,dxdy\Bigr)^{1/2} \\
+c_1\iint w^2_x\rho_{2\alpha,2\beta}\,dxdy.
\end{multline*}
As a result, the statement of the theorem follows from inequality \eqref{4.7}.
\end{proof}

\begin{remark}\label{R4.1}
Similarly to the methods of \cite{F89} one can obtain uniqueness and continuous dependence of weak solutions to problem \eqref{1.1}--\eqref{1.3} in the space $L_2^\alpha$ for $\alpha\geq\displaystyle{\frac{1}{2}(1+\frac{1}{b})}$ under the rate growth restriction 
$|g'(u)|\leq c|u|^b$ for $b\in(0,1)$, which, unfortunately, excludes Zakharov--Kuznetsov equation itself.
\end{remark}

\end{document}